\documentclass[12pt]{article}\usepackage{geometry, enumerate, amsmath, amssymb, amsthm, amscd, hyperref, graphicx, color, enumerate, mathrsfs}
\usepackage[all]{xy}
\usepackage[pagewise]{lineno}

\title{Polynomial Dedekind domains with finite residue fields of prime characteristic}

\date{\today}

\author{Giulio Peruginelli\footnote{Department of Mathematics "Tullio Levi-Civita", University of Padova, Via Trieste, 63, 35121 Padova, Italy. E-mail: gperugin@math.unipd.it}
}

\numberwithin{equation}{section}

\newtheorem{Prop}[equation]{Proposition}
\newtheorem{Thm}[equation]{Theorem}
\newtheorem{Lem}[equation]{Lemma}
\newtheorem{Cor}[equation]{Corollary}

\theoremstyle{definition}\newtheorem{Def}[equation]{Definition}

\newtheorem{Rem}[equation]{Remark}
\theoremstyle{definition}

\newcommand{\Q}{\mathbb{Q}}
\newcommand{\N}{\mathbb{N}}
\newcommand{\Z}{\mathbb{Z}}

\newcommand{\PP}{\mathbb{P}}
\newcommand{\Int}{\textnormal{Int}}
\newcommand{\IntQ}{\Int_{\Q}}
\newcommand{\olK}{\overline{K}}

\newcommand{\olD}{\overline{D}}

\newcommand{\oZp}{\overline{\Z_p}}
\newcommand{\Spec}{\text{Spec}}
\newcommand{\hZ}{\widehat{\mathbb{Z}}}
\newcommand{\ohZ}{\overline{\hZ}}
\newcommand{\Pic}{\text{Pic}}
\newcommand{\uE}{\underline{E}}
\newcommand{\Gal}{\text{Gal}}
\newcommand{\oQp}{\overline{\Q_p}}

\begin{document}
\leftmark{\noindent  accepted for publication in the Pacific Journal of Math (2023).}
{\let\newpage\relax\maketitle} 

\begin{abstract}
We show that every Dedekind domain $R$   lying between the polynomial rings $\Z[X]$ and $\Q[X]$ with the property that its residue fields of prime characteristic are finite fields 
is equal to a generalized ring of integer-valued polynomials, that is, for each prime $p\in\Z$ there exists a finite subset $E_p$ of transcendental elements over $\Q$ in  the absolute integral closure $\overline{\Z_p}$ of the ring of $p$-adic integers  such that $R=\{f\in\Q[X]\mid f(E_p)\subseteq \overline{\Z_p},\forall \text{ prime }p\in\Z\}$. Moreover, we prove that the class group of $R$ is isomorphic to a direct sum of a countable family of  finitely generated abelian groups. Conversely, any group of this kind is the class group of a Dedekind domain $R$ between $\Z[X]$ and $\Q[X]$.\\

{\small \noindent Keywords: Dedekind domain, Class group, Integer-valued polynomials.

\noindent MSC Primary: 13F20, 13F05, 13B25, 20K99.}

\end{abstract}

\begin{flushright}
\vspace{-0.5cm}{\small\emph{to the everlasting memory of Robert Gilmer}}
\end{flushright}

\section{Introduction}

Given a Dedekind domain $D$, the class group of $D$  measures how far $D$ is from being a UFD and it is therefore an important object in the study of factorization problems in the ring $D$.  It is well-known that the class group of the ring of integers of a number field is a finite abelian group. In contrast with this result,  Claborn proved in \cite{Claborn} the groundbreaking result that every abelian group occurs as the class group of a suitable Dedekind domain. 

After a few years, Eakin and Heinzer showed in \cite{EakHei}  that every finitely generated abelian group is the class group of a Dedekind domain between $\Z[X]$ and $\Q[X]$.   More generally, they proved that if $V_1,\ldots,V_n$ are  distinct DVRs with same quotient field $K$ and for each $i=1,\ldots,n$, $\{V_{i,j}\}_{j=1}^{g_i}$ is a finite collection of DVRs extending $V_i$ to $K(X)$, each of which is  residually algebraic over $V_i$ (i.e., the extension of the residue fields is algebraic), then the following is a Dedekind domain:
$$R=\bigcap_{i,j}V_{i,j}\cap K[X].$$
Moreover, they also give an explicit description of the class group of such a domain $R$. By means of this result, they were able to show the aforementioned result considering suitable residually algebraic extensions of a finite set of DVRs  of $\Q$ to $\Q(X)$.

Actually, if we suppose that each  residue field extension of  $V_{i,j}$ over $V_i$  is  finite, a ring $R$ constructed as above can be represented as a ring of integer-valued polynomials in the following way. For each $i,j$, by \cite[Theorem 2.5 \& Proposition 2.2]{PerTransc}, there exists an element $\alpha_{i,j}$ in  the algebraic closure $\overline{\widehat{K}_i}$ of the $V_i$-adic completion $\widehat{K}_i$ of $K$, $\alpha_{i,j}$ transcendental over $K$, such that $V_{i,j}=V_{i,\alpha_{i,j}}=\{\varphi\in K(X)\mid \varphi(\alpha_{i,j})\in\overline{\widehat{V_i}}\}$, where $\overline{\widehat{V_i}}$ is the absolute integral closure of $\widehat{V_i}$, the completion of $V_i$. Hence,  the above ring $R$ can be represented as $R=\{f\in K[X]\mid f(\alpha_{i,j})\in\overline{\widehat{V_i}},\forall i,j\}$ (for more details, see the following section and also \cite[Remark 2.8]{PerTransc}).

More recently, Glivick\'y and \u{S}aroch  in \cite{GliSa} investigated a family of quasi-euclidean subrings of $\Q[X]$ depending on a parameter $\alpha\in\hZ$, the profinite completion of $\Z$. A ring of this family is always a B\'ezout domain (i.e., finitely generated ideals are principal) and might be a PID or not, according to the finiteness of some set of primes depending on $\alpha$ and the set of polynomials in $\Z[X]$. It has been observed in \cite{GliSgaSa} that these rings can be realized as overrings of the classical ring of integer-valued polynomials $\Int(\Z)=\{f\in\Q[X] \mid f(\Z)\subseteq\Z\}$, which is a two-dimensional non-noetherian Pr\"ufer domain; such overrings have been completely characterized in \cite{ChabPer}. We will review this representation in the following section.

In the same area, Chang generalized Eakin and Heinzer's result in \cite{Chang}   proving that there exists an almost Dedekind domain $R$ (i.e., $R_M$ is a DVR for each maximal ideal $M$ of $R$) which is not Noetherian, lies between $\Z[X]$ and $\Q[X]$ and with class group isomorphic to a direct sum of a prescribed countable family of finitely generated abelian groups. As before, assuming the finiteness of the residue field extensions of the involved DVRs, Chang's  construction falls in the class of  integer-valued polynomial rings that we consider in this paper. It is worth mentioning that the Picard group of $\Int(\Z)$ is a free abelian group of countably infinite rank (\cite{GHLS}).

Here, we provide a complete description of the class of Dedekind domains $R$ lying between $\Z[X]$ and $\Q[X]$ such that their residue fields of prime characteristic are finite fields. Throughout the paper, for short we denote the last property by saying that  $R$ has finite residue fields of prime characteristic. We remark that the residue fields of such a domain $R$ cannot be all finite fields. In fact, since $R\subseteq\Q[X]_{(q)}$, for every irreducible $q\in\Q[X]$, then the residue field of the center of the DVR $\Q[X]_{(q)}$ on $R$ is a finite extension of $\Q$, hence  an infinite field. However, since $R$ is supposed to be Dedekind (in particular, a Pr\"ufer domain) the residue fields of prime characteristic are algebraic extensions of the corresponding prime field (see for example \cite[Theorem 3.14]{PerPrufer}). Infinite algebraic extensions of the prime fields of prime charactestic are also allowed, and that will be the content of  future research in this topic.

The paper is organized as follows. We first set the notation we will use throughout the paper and introduce the class of \emph{generalized ring of integer-valued polynomials}, which are subrings of $\Q[X]$ formed by  polynomials which are simultaneously integer-valued over  different subsets of integral elements over $\Z_p$, the ring of $p$-adic integers, for $p$ running over the set of integer primes. In section \ref{PolDedekind},  we review Loper and Werner's construction of Pr\"ufer domains in \cite{LopWer} and recall  that it falls into the class of generalized rings of integer-valued polynomials, as already observed in \cite[Remark 2.8]{PerTransc}. We then characterize when a ring of their construction is  a Dedekind domain  in Theorem \ref{DedekindInt}. In order to accomplish this objective, we introduce the definition of \emph{polynomially factorizable} subsets $\uE$ of $\ohZ=\prod_p \overline{\Z_p}$ (we refer to \S \ref{notation} for the unexplained notation), which turns out to be the key assumption for such a ring to be of finite character (hence, a noetherian Pr\"ufer domain, thus Dedekind). Furthermore, we show  in Theorem \ref{RDedekind}  that every Dedekind domain  $R$ with finite residue fields of prime characteristic lying between $\Z[X]$ and $\Q[X]$  is equal to a generalized ring of integer-valued polynomials with class group equal to a direct sum of a countable family of finitely generated abelian groups. Among other things, we will also characterize the PIDs among these class of domains, generalizing the aforementioned work of \cite{GliSa} (see also \cite{GliSgaSa}). We will also give a criteria for when two such generalized rings of integer-valued polynomials are equal. Finally, in section \ref{construction}, by means of a suitable modification of Chang's construction, given a group $G$ which is the direct sum of a countable family of finitely genereated abelian groups, we prove that there exists a Dedekind domain $R$  with finite residue fields of prime characteristic, $\Z[X]\subset R\subseteq\Q[X]$, with class group $G$, thus giving a positive answer to a question raised by Chang in \cite{Chang}. Note that, by the previous results, such a domain is a generalized ring of integer-valued polynomials.

It has come to our attention that Theorem 7 of \cite{ChangGer} shows the existence of a Dedekind domain with class group equal to a direct sum of a countable family of prescribed finitely generated abelian groups. However, that construction is based on a polynomial ring with an infinite set of indeterminates and there are prime ideals with infinite residue field.

\subsection{Notation}\label{notation}

The generalized rings of integer-valued polynomials considered in this paper falls into the class of integer-valued polynomials on algebras (see for example \cite{Frisch, FrischCorr, PerWerNontrivial}), which encompasses also the classical definition of ring of integer-valued polynomials. We now recall the latter definition. Let $D$ be an integral domain with quotient field $K$ and $A$ a torsion-free $D$-algebra such that $A\cap K=D$. We may evaluate polynomials $f\in K[X]$ at any element $a\in A$ inside the extended algebra $A\otimes_D K$. The $D$-algebra $A$ clearly embeds into $A\otimes_D K$ and if $f(a)\in A$ we say that $f$ is integer-valued at $a$. In general, given a subset $S$ of $A$, we define the ring of integer-valued polynomials over $S$:
$$\Int_K(S,A)=\{f\in K[X] \mid f(s)\in A,\forall s\in S\}.$$
Note that when $A=D$ we get the usual definition of ring of integer-valued polynomials  on a subset $S$ of  $D$, and in that case we omit the subscript $K$.  If $S=D=A$, then we set $\Int(D,D)=\Int(D)$.

For an integral domain $D$, we define the Picard group of $D$, denoted by $\text{Pic}(D)$, as the quotient of the abelian group of the  invertible fractional ideals of $D$ by the subgroup generated by the nonzero principal fractional ideals, where the operation is the ideal multiplication  (see \cite[\S VIII.1]{CaCh}). If $D$ is a Dedekind domain, then $\text{Pic}(D)$ is the usual  ideal  class group of $D$.

Let $\PP$ be the set of all prime numbers. For a fixed $p\in\PP$, we adopt the following notations:
\begin{itemize}
\item[-] $\Z_{(p)}$, the localization of $\Z$ at $p\Z$.
\item[-] $\Z_p,\Q_p$, the ring of $p$-adic integers and the field of $p$-adic numbers, respectively.
\item[-] $\overline{\Q_p},\overline{\Z_p}$, a fixed algebraic closure of $\Q_p$ and the absolute integral closure of $\Z_p$, respectively.
\item[-] for a finite extension $K$ of $\Q_p$, we denote by $O_K$ the ring of integers of $K$.
\item[-] $v_p$ denotes the unique extension of the $p$-adic valuation on $\Q_p$ to $\overline{\Q_p}$.
\item[-] If $\alpha\in\oQp$, we denote the ramification index $e(\Q_p(\alpha)\mid\Q_p)$ by $e_{\alpha}$.
\item[-] $\hZ=\prod_{p\in\PP}\Z_p$ the profinite completion of $\Z$.
\item[-] $\overline{\hZ}=\prod_{p\in\PP}\overline{\Z_p}$.
\item[-] for $\alpha\in\overline{\Q_p}$, we set
$$V_{p,\alpha}=\{\varphi\in \Q(X)\mid \varphi(\alpha)\in\overline{\Z_p}\}.$$
\end{itemize}
Clearly, $V_{p,\alpha}$ is a valuation domain of $\Q(X)$ extending $\Z_{(p)}$ with maximal ideal equal to $M_{p,\alpha}=\{\varphi\in V_{p,\alpha}\mid v_p(\varphi(\alpha))>0\}$. Moreover, $V_{p,\alpha}$ is a DVR if $\alpha$ is transcendental over $\Q$ and it has rank $2$ otherwise. In the former case, the ramification index $e(V_{p,\alpha}\mid \Z_{(p)})$ is equal to  $e_{\alpha}$.  In either case, let $O_{\alpha}$ and $M_{\alpha}$ be the valuation domain and maximal ideal of $\Q_p(\alpha)$, respectively. Then, the residue field of $V_{p,\alpha}$ is equal to $O_{\alpha}/M_{\alpha}$ and $pO_{\alpha}=M_{\alpha}^e$, for some integer $e$, which is equal to $e_{\alpha}$ ( for all these results, see  \cite[Proposition 2.2 \& Theorem 2.5]{PerTransc}).

The following result mentioned in the introduction  characterizes  residually algebraic extensions of $\Z_{(p)}$ to $\Q(X)$ of a certain kind; the valuation overrings of the Dedekind domains we are dealing with belong to this class.
\begin{Thm}\label{extension Vpa}\cite[Theorem 2.5 \& Theorem 3.2]{PerTransc}
Let $W\subset \Q(X)$ be a valuation domain with maximal ideal $M$ extending $\Z_{(p)}$ for some $p\in\PP$. If $pW=M^e$ for some $e\geq 1$ and $W/M\supseteq\Z/p\Z$ is a finite extension, then there exists $\alpha\in\overline{\Q_p}$ such that $W=V_{p,\alpha}$.  
Moreover, for $\alpha,\beta\in\oQp$, we have $V_{p,\alpha}=V_{p,\beta}$ if and only if $\alpha,\beta$ are conjugate over $\Q_p$.
\end{Thm}
 Clearly, if $W$ is as in the assumptions of Theorem \ref{extension Vpa} and $\Z[X]\subset W$, then $\alpha\in\oZp$.

Given $f\in\Q[X]$, the evaluation of $f(X)$ at an element  $\alpha=(\alpha_p)\in\overline{\hZ}$ is done componentwise:
$$f(\alpha)=(f(\alpha_p))\in\prod_{p\in\PP}\overline{\Q_p}.$$
We say that $f$ is \emph{integer-valued} at $\alpha$ if $f(\alpha)\in\overline{\hZ}$, which is equivalent to  $f\in V_{p,\alpha_p}$ for all $p\in\PP$. 

\begin{Def}
Given a subset $\uE$ of $\ohZ$, we define the \emph{generalized ring of integer-valued polynomials on $\uE$} as:
$$\IntQ(\uE,\ohZ)=\{f\in\Q[X]\mid f(\alpha)\in\ohZ,\forall \alpha\in\uE\}.$$
\end{Def}
If $\uE=\hZ$, then  $\IntQ(\hZ,\ohZ)=\IntQ(\hZ,\hZ)=\Int(\Z)$; in fact, the first equality follows easily from the fact that the polynomials have rational coefficients, for the last equality see \cite[Remark 6.4]{ChabPer} (essentially, $\Z$ is dense in $\hZ$).  We recall that the family of overrings of $\Int(\Z)$ which are  contained in $\Q[X]$ is formed exactly by the rings $\Int_{\Q}(\uE,\hZ)$, as $\underline E$ ranges through the subsets of $\hZ$ of the form $\prod_{p\in\PP} E_p$, where for each prime $p$, $E_p$ is a closed (possibly empty) subset of $\Z_p$ (see \cite[Theorem 6.2]{ChabPer}). In the study of a generalized ring of integer-valued polynomials $\Int_{\Q}(\uE,\ohZ)$, without loss of generality we may suppose that the subset $\uE$ of $\ohZ$ is of the form $\uE=\prod_{p\in\PP} E_p$ (see the arguments given in \cite[Remark 6.3]{ChabPer}). Note that we allow each component  $E_p$ of $\uE$ to be equal to the empty set.
\vskip1cm

\section{Polynomial Dedekind domains}\label{PolDedekind}

Loper and Werner exhibit a construction of Pr\"ufer domains  between $\Z[X]$ and $\Q[X]$ in \cite{LopWer} in order to show the existence of a Pr\"ufer domain strictly contained in $\Int(\Z)$.  As in the work of Eakin and Heinzer \cite{EakHei} mentioned in the Introduction, their construction is obtained by intersecting a suitable family of valuation domains of $\Q(X)$ indexed by $\PP$ with $\Q[X]$. A valuation domain of this family is  equal to $V_{p,\alpha}$, for some $\alpha\in\overline{\Z_p}$, by Theorem \ref{extension Vpa}  and the fact that $X$ is in  every valuation domain of this family. As it has already been observed in \cite[Remark 2.8]{PerTransc}, a ring of the Loper and Werner's construction can be represented as a generalized ring of integer-valued polynomials $\IntQ(\uE,\ohZ)$, for a suitable subset $\uE$ of $\ohZ$ which satisfies the following definition.
\begin{Def}
Let $\uE=\prod_{p\in\PP} E_p\subset\ohZ$. We say that $\uE$ is locally bounded, if, for each prime $p$, $E_p$ is a subset of $\overline{\Z_p}$ of bounded degree, that is, $\{[\Q_p(\alpha):\Q_p]\mid\alpha\in E_p\}$ is bounded.
\end{Def}
As we have already said above, some of the components $E_p$ of $\uE$ may be equal to the empty set. Since $\Q_p$ has at most finitely many extensions of degree bounded by some fixed positive integer, if $E_p\subset\overline{\Z_p}$ has bounded degree then   $E_p$ is contained in a finite extension of $\Q_p$. 

By Theorem \ref{extension Vpa}, a Pr\"ufer domain constructed in \cite{LopWer} can be represented as  the following intersection of valuation domains (see also  \cite{ChabPer}), where $\uE=\prod_{p\in\PP}E_p\subset \ohZ$ is locally bounded: 
\begin{equation}\label{IntQEZrepr}
\IntQ(\uE,\ohZ)=\bigcap_{p\in\PP}\bigcap_{\alpha_p\in E_p}V_{p,\alpha_p}\cap\bigcap_{q\in\mathcal{P}^{\text{irr}}}\Q[X]_{(q)}
\end{equation}
 where  $\mathcal{P}^{\text{irr}}$ denotes the set of irreducible polynomials in $\Q[X]$; note that $\cap_{q\in \mathcal{P}^{\text{irr}}}\Q[X]_{(q)}=\Q[X]$. Similarly, for the ring $\IntQ(E_p,\overline{\Z_p})=\{f\in\Q[X]\mid f(E_p)\subseteq\overline{\Z_p}\}$ we have:
\begin{equation}\label{IntQEpZprepr}
\IntQ(E_p,\overline{\Z_p})=\bigcap_{\alpha_p\in E_p}V_{p,\alpha_p}\cap\bigcap_{q\in\mathcal{P}^{\text{irr}}}\Q[X]_{(q)}.
\end{equation}
In particular,  $\IntQ(\uE,\ohZ)=\bigcap_{p\in\PP}(\Z\setminus p\Z)^{-1}\IntQ(\uE,\ohZ)=\bigcap_{p\in\PP}\IntQ(E_p,\overline{\Z_p})$ by Lemma \ref{localization}. 

By means of the representation \eqref{IntQEZrepr}, the main result of Loper and Werner (\cite[Corollary 2.12]{LopWer})  can now be restated as follows:
\begin{Thm}\label{locally bounded Prufer}
Let $\uE\subset\ohZ$ be locally bounded. Then  the ring $\IntQ(\uE,\ohZ)$ is a Pr\"ufer domain.
\end{Thm}
We want to characterize when a  ring of the form $\IntQ(\uE,\ohZ)$, $\uE\subseteq\ohZ$, is a Dedekind domain. In order to accomplish this objective, we need to describe  the prime spectrum of this ring  when $E$ is locally bounded. It is customary for ring of integer-valued polynomials to distinguish the prime ideals into two different kinds, and we do the same here in our setting: given a prime ideal $P$ of $R=\IntQ(\uE,\ohZ)$, we say that $P$ is a \emph{non-unitary} if  $P\cap\Z=(0)$ and that $P$ is \emph{unitary} if $P\cap \Z=p\Z$ for some $p\in\PP$. 

It is a classical result that each non-unitary prime ideal of $R$ is equal to:
$$\mathfrak{P}_q=q(X)\Q[X]\cap R$$
for some $q\in \mathcal{P}^{\text{irr}}$ (see for example \cite[Corollary V.1.2]{CaCh}). 

If $P\cap\Z=p\Z$, $p\in\PP$, and $\alpha\in E_p$, the following is a unitary prime ideal of $R$:
$$\mathfrak{M}_{p,\alpha}=\{f\in R\mid v_p(f(\alpha))>0\}.$$
If $E_p$ is a closed subset of $\overline{\Z_p}$ for each prime $p$, and $\uE=\prod_p E_p$ is locally bounded, we are going to show that each unitary prime ideal of $R$ is equal to $\mathfrak{M}_{p,\alpha}$, for some $p\in\PP$ and $\alpha\in E_p$.

\begin{Lem}\label{localization}
 Let $\uE\subseteq\ohZ$ be any subset,  $P$ be a finite subset of $\PP$ and $S$ the multiplicative subset of $\Z$ generated by $\PP\setminus P$. Then $S^{-1}\IntQ(\uE,\ohZ)=\bigcap_{p\in P}\IntQ(E_p,\overline{\Z_p})$.

In particular, for each $p\in\PP$, $(\Z\setminus p\Z)^{-1}\IntQ(\uE,\ohZ)=\IntQ(E_p,\overline{\Z_p})$.
\end{Lem}
\begin{proof}
The proof follows by an argument similar to the one of  \cite[Proposition 4.2]{ChabPer1}. Let $R=\IntQ(\uE,\ohZ)$ and $R_p=\IntQ(E_p,\overline{\Z_p})$, for each $p\in P$. The containment $S^{-1}R\subseteq\bigcap_{p\in P}R_p$ is clear, since $R\subseteq R_p$ and for every $d\in S$, $d$ is a unit in  $R_p$, for each $p\in P$. Conversely, let $f\in \bigcap_{p\in P}R_p$. Let $d\in\Z$, $d\not=0$, be such that $df\in\Z[X]$ and let $d=t\prod_{p\in P}p^{a_p}$, $a_p\geq0$ and $t\in\Z$ not divisible by any $p\in P$. Then, letting $g=tf$, we have that $g$ is in $\Z_{(q)}[X]\subset R_q$ for each $q\notin P$ and $g$ is in $R_p$ for each $p\in P$ because $t$ is a unit in $\Z_{(p)}$, $\forall p\in P$. Hence, $f=\frac{g}{t}\in S^{-1}R$, as wanted.
\end{proof}

\begin{Prop}\label{prime ideals}
Let $\uE=\prod_p E_p\subset\ohZ$ be locally bounded and closed. If $M\subset \IntQ(\uE,\ohZ)$ is a unitary prime ideal such that $M\cap\Z=p\Z$ for some $p\in\PP$, then $M$ is maximal and there exists $\alpha\in E_p$ such that $M=\mathfrak{M}_{p,\alpha}$.
\end{Prop}
\begin{proof}
Let $R=\IntQ(\uE,\ohZ)$.  We use the fact that $R$ is a Pr\"ufer domain by Theorem \ref{locally bounded Prufer}.

Let $M$ be a  unitary prime  ideal of $R$ and let $V=R_M$. Then, by Lemma \ref{localization}, we have $R_p=\IntQ(E_p,\overline{\Z_p})\subset V$, since $(\Z\setminus p\Z)^{-1}V=V$. Let $M'$ be the center of $V$ on $R_p$. Since  $M'\cap R=M$, it is sufficient to show that $M'=\mathfrak{M}_{p,\alpha}=\{f\in R_p\mid v_p(f(\alpha))>0\}$, for some $\alpha\in E_p$ (with a slight abuse of notation, we denote the unitary prime ideals of $R$ and $R_p$ in the same way). Let $f\in R_p$. Let $K$ be a finite extension of $\Q_p$ such that $O_K$ contains $E_p$ and let $i_0,\ldots,i_{q-1}\in O_K$ be a set of representatives for $O_K/\pi O_K\cong\mathbb{F}_q$, where $\pi$ is a uniformizer of $O_K$ (i.e., a generator of the maximal ideal of $O_K$). For each $\alpha\in E_p$, there exists some $j\in\{0,\ldots,q-1\}$ such that $f(\alpha)-i_j\in \pi O_K$. In particular, $\prod_{j=0}^{q-1}(f(\alpha)-i_j)\in \pi O_K$ for each $\alpha\in E_p$. Note that the polynomials $X^q-X$ and $\prod_{j=0}^{q-1}(X-i_j)$ coincide modulo $\pi$, so in particular $f(\alpha)^q-f(\alpha)\in\pi O_K$.  If $e=e(O_K\mid \Q_p)$, we have $(f(\alpha)^q-f(\alpha))^e\in p O_K$. Equivalently, $(f^q-f)^e\in pR_p$, which is contained in $M'$. Since $M'$ is a prime ideal, it follows that $f^q-f\in M'$, so modulo $M'$, $f$ satisfies the equation $X^q-X=0$. This shows that $R_p/M'$ is contained in the finite field $\mathbb{F}_q$, so it is a finite domain, hence a field. This proves that $M'$ is maximal.  Note that, since $R/M\subseteq R_p/M'$ and the latter is a finite field, it follows also that $M$ is a maximal ideal of $R$.

Since $R_p$ is countable, $M'$ is countably generated, say $M'=\bigcup_{n\in\N}I_n$, where $I_n=(p,f_1,\ldots,f_n)$ for each $n\in\N$. By \cite[Proposition 1.4]{GilmHeinz}, for each $n\in\N$, there exists $\alpha_n\in E_p$ such that $I_n\subset \mathfrak{M}_{p,\alpha_n}$ (we may exclude the non-unitary prime ideals of $R_p$ because they do not contain $p$, hence neither $I_n$ for every $n$). Suppose first that $E_p$ is finite. Then there exists $\alpha\in E_p$ such that the set $J=\{n\in\N\mid I_n\subset\mathfrak{M}_{p,\alpha}\}$ is a cofinal subset of $\N$. Hence, for each $f\in M'$, there exists $n\in J$ such that $f\in I_n\subset\mathfrak{M}_{p,\alpha}$, so that $M'\subseteq \mathfrak{M}_{p,\alpha}$ and therefore equality holds since $M'$ is maximal. If $E_p$ is infinite, since it is a  closed subset (because $\uE$ is closed) contained in a finite extension of $\Q_p$, by compactness we may extract a sequence $\{\alpha_n\}_{n\in\N}$ from $E_p$ converging to some element $\alpha\in E_p$. Without loss of generality we suppose that $\alpha_n\to\alpha$. Now, for each $f\in M'$, $f\in I_n\subset\mathfrak{M}_{p,\alpha_n}$ for some $n$. Since $I_n\subseteq I_{n+1}$ for each $n\in\N$, $f\in \mathfrak{M}_{p,\alpha_m}$ for each $m\geq n$, that is $v_p(f(\alpha_m))>0$. By continuity we get that $v_p(f(\alpha))>0$, that is, $f\in\mathfrak{M}_{p,\alpha}$. Therefore as before we conclude that $M'=\mathfrak{M}_{p,\alpha}$.
\end{proof}

In particular, if $\IntQ(\uE,\ohZ)$ is a Pr\"ufer domain, given a maximal unitary ideal $\mathfrak{M}_{p,\alpha}$, $p\in\PP$ and $\alpha\in E_p$, we have 
\begin{equation}\label{localization unitary}
\IntQ(\uE,\ohZ)_{\mathfrak{M}_{p,\alpha}}=V_{p,\alpha}.
\end{equation}
Similarly, for $q\in\mathcal{P}^{\text{irr}}$, we have
\begin{equation}\label{localization non-unitary}
\IntQ(\uE,\ohZ)_{\mathfrak{P}_q}=\Q[X]_{(q)}.
\end{equation}
We call the valuation domains $V_{p,\alpha}$ unitary, and the others $\Q[X]_{(q)}$ non-unitary. Similar equalities hold for the Pr\"ufer domain $\IntQ(E_p,\overline{\Z_p})$.  Note that the residue field of $\IntQ(\uE,\ohZ)$ at a unitary prime ideal is a finite field (by the property of the unitary valuation overrings we discussed about in \S \ref{notation}), while the residue field of a non-unitary prime ideal of $\IntQ(\uE,\ohZ)$  is a finite extension of the rationals, hence an infinite field. 

 We finish this section with the following remark.
\begin{Rem}
By Theorem \ref{extension Vpa}, given a ring $\IntQ(E_p,\oZp)$,  without loss of generality we may assume that the elements of $E_p$ are pairwise non-conjugate over $\Q_p$. Under this further assumption and if $E_p$ is bounded (i.e., contained in a finite extension of $\Q_p$), Theorem \ref{locally bounded Prufer}, equation \eqref{localization unitary} and Proposition \ref{prime ideals} imply that there is a one-to-one correspondence between the elements of $E_p$ and the unitary valuation overrings $V_{p,\alpha_p},\alpha_p\in E_p$, of $\IntQ(E_p,\oZp)$.
\end{Rem}

\subsection{Local case}

For a fixed $p\in\PP$, we characterize in this section the subsets $E_p$ of $\overline{\Z_p}$ for which the corresponding ring of integer-valued polynomials $\IntQ(E_p,\overline{\Z_p})$ is a Dedekind domain. The following  proposition  is a generalization of  \cite[Theorem 4.3 (2)]{Chang}.

\begin{Prop}\label{finiteZp}
Let $E_p$ be a subset of $\overline{\Z_p}$. Then  $\IntQ(E_p,\overline{\Z_p})$ is a Dedekind domain  with finite residue fields of prime characteristic  if and only if $E_p$ is a finite subset of transcendental elements over $\Q$. 

 Suppose that $E_p=\{\alpha_1,\ldots,\alpha_n\}$  and the $\alpha_i$'s are pairwise non-conjugate over $\Q_p$. Then, then the class group of $\Int_{\Q}(E_p,\overline{\Z_p})$ is isomorphic to $\Z/e\Z\oplus\Z^{n-1}$, where $e=\text{g.c.d.}\{e_{\alpha_i}\mid i=1,\ldots,n\}$. In particular, $\Int_{\Q}(E_p,\overline{\Z_p})$ is a PID if and only if $E_p$ contains at most one  element  $\alpha_p\in\oZp$, such that $\alpha_p$ is transcendental over $\Q$ and unramified over $\Q_p$. 
\end{Prop}
\begin{proof}
Let $R_p=\IntQ(E_p,\overline{\Z_p})$. Note that, if $E_p$ is the empty set, then $R_p=\Q[X]$. We assume henceforth that $E_p\not=\emptyset$.

Suppose $R_p$ is a Dedekind domain  with finite residue fields of prime characteristic. We show first that each maximal unitary ideal $M$ of $R_p$ is equal to $\mathfrak{M}_{p,\alpha_p}$, for some $\alpha_p\in E_p$. Let $V$ be  a unitary valuation overring of $R_p$ which is centered on $M$. By Theorem \ref{extension Vpa}, there exists $\alpha_0\in\overline{\Z_p}$ such that $V=V_{p,\alpha_0}$. Then, $M=\mathfrak M_{p,\alpha_0}$. Since $M$ is finitely generated and $R_p$ is Pr\"ufer, by \cite[Proposition 1.4]{GilmHeinz} $M\subseteq \mathfrak M_{p,\alpha_p}$ for some $\alpha_p\in E_p$ (we may exclude the non-unitary prime ideals of $R_p$ because they do not contain $p$, hence neither $M$). Since $M$ is maximal, it follows that $M=\mathfrak M_{p,\alpha_p}$, which means that $\alpha_0$ and $\alpha_p$ are conjugate over $\Q_p$  by \cite[Theorem 3.2]{PerTransc}. Hence, without loss of generality, we may suppose that $\alpha_0\in E_p$.  Note that each $\alpha_p\in E_p$ is transcendental over $\Q$, otherwise the valuation overring $V_{p,\alpha_p}$ of $R_p$ would have rank $2$. Since $R_p$ is Dedekind, $p$ is contained in only finitely many maximal ideals of this ring; necessarily, such ideals are unitary. By the previous argument, such ideals are equal to $\mathfrak M_{p,\alpha_p}$, for $\alpha_p\in E_p$. Since by  Theorem \ref{extension Vpa} and \eqref{localization unitary}, $\mathfrak M_{p,\alpha_p}=\mathfrak M_{p,\beta_p}$ if and only if $\alpha_p,\beta_p\in E_p$ are conjugate over $\Q_p$, it follows
that $E_p$ is a finite subset of $\oZp$.

Conversely, suppose now that $E_p\subset\overline{\Z_p}$ is a finite subset of transcendental elements over $\Q$.  The fact that $\IntQ(E_p,\overline{\Z_p})$ is a Dedekind domain follows from \cite[Theorem]{EakHei}, but we give a different self-contained argument  based on the previous results. In particular, $E_p$ has bounded degree, so, by  Theorem \ref{locally bounded Prufer},  $R_p$ is Pr\"ufer. By \eqref{IntQEpZprepr}, $R_p$ is equal to an intersection of DVRs which are essential over it. Moreover, each non-zero $f\in R_p$ belongs to finitely many maximal ideals, since $E_p$ is finite and $f$ has finitely many irreducible factors in $\Q[X]$. Hence, $R_p$ is a Krull domain, so, by  \cite[Theorem 43.16]{Gilm}, $R_p$ is a Dedekind domain. Finally,  $R_p$ has  finite residue fields of prime characteristic, because each of the  unitary  valuation overring of $R_p$ (namely, $V_{p,\alpha_p}$, $\alpha_p\in E_p$)  has finite residue field.

 Assuming that the elements of $E_p$ are pairwise non-conjugate over $\Q_p$, the claim regarding the class group follows easily from \cite[Theorem]{EakHei}, taking into account the  representation \eqref{IntQEpZprepr}. If $E_p=\{\alpha_1,\ldots,\alpha_n\}$, let $\bold{e}=(e_{\alpha_1},\ldots,e_{\alpha_n})\in\Z^n$ and $e=GCD(e_{\alpha_1},\ldots,e_{\alpha_n})$. Then, the class group of $R_p$ is isomorphic to
$$\Z^{n}/<\bold{e}>\cong\Z/e\Z\oplus\Z^{n-1}.$$
The last claim follows at once from the description of the class group.
\end{proof}

\subsection{Global case}\label{generalized not Noetherian}

If, for each $p\in\PP$,  $E_p\subset\overline{\Z_p}$ is a finite subset of transcendental elements over $\Q$  and $\uE=\prod_p E_p$, then, by \cite[Corollary 2.6]{Chang}, $\IntQ(\uE,\ohZ)$ is an almost Dedekind domain. However, this  ring might not be Noetherian, that is, a Dedekind domain. See for example the construction of \cite[Theorem 3.1]{Chang}, in which the polynomial $X$ is divisible by infinitely many primes $p\in\PP$. In general, an almost Dedekind domain $R$ is Dedekind if and only if it has finite character, that is, each non-zero $f\in R$ belongs to finitely many maximal ideals of $R$ (\cite[Theorem 37.2]{Gilm}), or, equivalently, $v(f)\not=0$ only for finitely many valuation overrings $V$ of $R$ (which are only DVRs). We aim to characterize the subsets $\uE=\prod_p E_p$ of $\ohZ$ for which $\IntQ(\uE,\ohZ)$ is Dedekind. 

We introduce the following definition.
\begin{Def}
We say that $\uE$ is \emph{polynomially factorizable} if, for each $g\in\Z[X]$ and $\alpha=(\alpha_p)\in\uE$, there exist $n,d\in\Z$, $n,d\geq 1$ such that $\frac{g(\alpha)^n}{d}$ is a unit of $\ohZ$, that is, $v_p(\frac{g(\alpha_p)^n}{d})=0$, $\forall p\in\PP$.
\end{Def}
Note that $g(\alpha)^n=(g(\alpha_p)^n)\in\ohZ$. Loosely speaking, a subset $\uE$ of $\ohZ$ is polynomially factorizable if, for every $g\in \Z[X]$ and $\alpha\in\uE$, $g(\alpha)\in\ohZ$  is divisible only by finitely many   primes $p\in\PP$ (up to some exponent $n\geq1$), or, equivalently, all but finitely many components of $g(\alpha)$ are units.  Note that, if the above condition of the definition holds, then $g(\alpha)^n$ and $d$ generate the same principal ideal of $\ohZ$.

The next lemma gives a simple characterization of polynomially factorizable subsets $\uE$ of $\ohZ$ in terms of the finiteness of some sets of primes associated to every polynomial in $\Z[X]$.
For every $g\in\Z[X]$ and subset $\uE=\prod_p E_p\subseteq\ohZ$, we set
$$\PP_{g,\uE}=\{p\in\PP \mid \exists \alpha_{p}\in E_{p} \text{ such that }v_p(g(\alpha_{p}))>0\}.$$
The next result shows that $\uE$ is polynomially factorizable if and only if $\PP_{g,\uE}$ is finite for every $g\in\Z[X]$.

\begin{Lem}\label{pol factorizable}
Let $g\in\Z[X]$ and $\uE=\prod_p E_p\subset\ohZ$, where each  $E_p\subset\overline{\Z_p}$ is a closed set of  transcendental elements over $\Q$. Then the following conditions are equivalent:
\begin{itemize}
\item[i)] the set $\PP_{g,\uE}$ is finite.
\item[ii)] for each $\alpha\in\uE$, there exist $n,d\in\Z$, $n,d\geq1$  such that $\frac{g(\alpha)^n}{d}$ is a unit of $\ohZ$.
\end{itemize}
\end{Lem}
\begin{proof}
We use the following easy remark: for $\alpha=(\alpha_p)\in\hZ=\prod_p \Z_p$, the set $\{p\in\PP\mid v_p(\alpha_p)>0\}$ is finite if and only if there exists $d\in\Z$,  $d\geq 1$,  such that $\alpha\hZ=d\hZ$.

Suppose i) holds and let $\alpha=(\alpha_p)\in\uE$. By assumption, there are only finitely many $p\in\PP$ such that  $v_p(g(\alpha_p))>0$, for some $\alpha_p\in E_p$, say, $p_1,\ldots,p_k$. Let $\alpha\in\uE$ be fixed; in particular, there exists $n\in\N$ such that $nv_p(g(\alpha_p))=a_p\in\Z$ for each prime $p$ (where $a_p=0$ for all $p\notin\{p_1,\ldots,p_k\}$). Hence, if we let $d=\prod_{i=1}^k p_i^{a_{p_i}}$ we get $v_p(g(\alpha_p)^n)=v_p(d)$ for all $p\in\PP$, hence ii) holds.

Assume now that ii) holds and suppose that $\PP_{g,\uE}$ is infinite. For each $p\in\PP_{g,\uE}$, let $\alpha_p\in E_p$ be such that $v_p(g(\alpha_p))>0$ and consider the element $\alpha=(\alpha_p)\in \uE$, where $\alpha_p$ is any element of $E_p$ for $p\notin\PP_{g,\uE}$. If there is no $n\geq 1$ such that $nv_p(g(\alpha_p))=a_p\in\Z$ for all $p\in\PP$ we immediately get a contradiction. Suppose instead that such an $n$ exists. Since $a_p$ is non-zero for infinitely many $p\in\PP$, there is no $d\in\Z$ such that $v_p(\frac{g(\alpha_p)^n}{d})=0$ for each $p\in\PP$, which again is a contradiction.

The final sentence of the statement follows immediately from the above arguments. \end{proof}

\begin{Rem}
By Lemma \ref{pol factorizable}, it follows easily that  a subset $\uE\subseteq\ohZ$ is polynomially factorizable if and only if $\PP_{g,\uE}$ is finite for each irreducible $g\in\Z[X]$. In fact, if $g=\prod_i g_i$, where $g_i\in\Z[X]$ are irreducible, then $\PP_{g,\uE}=\bigcup_i \PP_{g_i,\uE}$.

It is well-known that, given a  nonconstant  $q\in\Z[X]$, there exist infinitely many $p\in\PP$ for which there exists $n\in\Z$ such that $q(n)$ is divisible by $p$ (see for example the proof of \cite[Prop. V.2.8]{CaCh}). In particular, $\hZ$ is not polynomially factorizable by Lemma \ref{pol factorizable}.
\end{Rem}

The next lemma describes the Picard group of  $\IntQ(\uE,\ohZ)$ in terms of the Picard groups of the localizations $\IntQ(E_p,\overline{\Z_p})$, $p\in \PP$ (see Lemma \ref{localization}).
\begin{Lem}\label{Picdirectsum}
Let $\underline{E}=\prod_p E_p\subset\overline{\hZ}$ be a subset. Then
$$\Pic(\IntQ(\uE,\ohZ))\cong\bigoplus_{p\in\PP}\Pic(\IntQ(E_p,\overline{\Z_p})).$$
\end{Lem}
\begin{proof}
Let $R=\IntQ(\uE,\ohZ)$ and $R_p=(\Z\setminus p\Z)^{-1}R$, for $p\in\PP$; by Lemma \ref{localization}, $R_p=\IntQ(E_p,\overline{\Z_p})$. Since the proof follows by the same arguments of \cite[Theorem 1]{GHLS}, we just sketch it and refer to the cited paper for the details. By a classical argument (see for example \cite[Lemma 1]{McQ}), every finitely generated ideal $J$ of $R$ (in particular, every invertible ideal of $R$) is isomorphic to a finitely generated unitary ideal $I$, that is, $I\cap\Z=d\Z\not=(0)$. For such an ideal,  $(I\cap\Z)_{(p)}=\Z_{(p)}$ for all $p\in\PP$ not dividing $d$,  so $IR_p=R_p$. This argument shows that we have a well-defined map from $\Pic(R)$ to $\bigoplus_{p\in\PP}\Pic(R_p)$.

If $I$ is a unitary ideal of $R$, say $I\cap\Z=d\Z$, such that $IR_p$ is principal, it is generated by $d$. Hence, $I$ and $dR$ have the same localizations at each prime $p\in\PP$, so they are equal. This shows that the previous map is injective.

For the surjectivity, it is sufficient to show that, if $J_p$ is an invertible unitary ideal of $R_p$, for some $p\in\PP$, then there exists an invertible ideal $J$ of $R$ such that $JR_p=J_p$ and $JR_q=R_q$ for each $q\in\PP\setminus\{p\}$. The ideal $J=J_p\cap R$ has the required properties.
\end{proof}

We may now characterize when a generalized ring of integer-valued polynomials $\IntQ(\uE,\ohZ)$ is Dedekind and describe its class group. 

\begin{Thm}\label{DedekindInt}
Let $\underline{E}=\prod_p E_p\subset\overline{\hZ}$ be a subset. Then $\IntQ(\uE,\ohZ)$ is a Dedekind domain  with finite residue fields of prime characteristic  if and only if $E_p$ is a finite set of transcendental elements over $\Q$ for each $p\in\PP$ and $\uE$ is polynomially factorizable.

In this case, the class group of $\Int_{\Q}(\underline{E},\hZ)$ is equal to a direct sum of a countable family of finitely generated abelian groups.
\end{Thm}
\begin{proof}
Let $R=\IntQ(\uE,\ohZ)$ and suppose the conditions for $\uE$ in the statement are satisfied. In particular, $\uE$ is locally bounded  and closed so, by Theorem \ref{locally bounded Prufer}, $R$ is Pr\"ufer. For $R$ to be Dedekind,  it is sufficient to show that it is a Krull domain (\cite[Theorem 43.16]{Gilm}). By assumption, each of the  unitary  valuation overrings of $R$ in the representation     \eqref{IntQEZrepr} is a DVR with finite residue field,  so $R$ has finite residue fields of prime characteristic by Proposition \ref{prime ideals}. We have to show that $R$ has finite character, that is, for each non-zero $f=\frac{g}{n}\in R$, $g\in\Z[X]$ and $n\in\Z\setminus\{0\}$, $f$ is contained in only finitely many maximal ideals of $R$. As in the proof of Proposition \ref{finiteZp}, $f$ is contained in only finitely many non-unitary prime ideals of $R$. We now check the maximal unitary ideals of $R$, described in the Proposition \ref{prime ideals}, which  contain $f$. Since the denominator $n$ of $f$ is divisible by only finitely many  $p\in\PP$, $f$ is contained in only finitely many maximal unitary ideals if and only if the same condition holds for $g$. Since $E_p$ is finite for each $p\in\PP$, this is equivalent to the finiteness of  the set $\PP_{g,\uE}$. Since $\uE$ is polynomially factorizable, by  Lemma \ref{pol factorizable}, $\PP_{g,\uE}$ is finite. 

Conversely, if $\IntQ(\uE,\ohZ)$ is a Dedekind domain  with finite residue fields of prime characteristic , then, for each prime $p$, the overring $\IntQ(E_p,\overline{\Z_p})$   is a Dedekind domain  with finite residue fields of prime characteristic  (\cite[Theorem 40.1]{Gilm}). By Proposition \ref{finiteZp}, $E_p$ is a finite subset of $\overline{\Z_p}$  formed by transcendental elements over $\Q$ (so, in particular, $\uE$ is locally bounded). If there exists some $g\in\Z[X]$ such that the set $\PP_{g,\uE}$ is infinite, then $g(X)$ would be contained in infinitely many unitary prime ideals of $\Int_{\Q}(\underline E,\overline{\hZ})$, a contradiction with \cite[Theorem 37.2]{Gilm}. Therefore, $\uE$ is polynomially factorizable by Lemma \ref{pol  factorizable}.

The final claim regarding the class group follows from Lemma \ref{Picdirectsum} and Proposition \ref{finiteZp}.
\end{proof}

The next corollary is a generalization of  \cite[Lemma 3.3]{GliSa}: it characterizes  the elements $\alpha$ in $\ohZ$ for which  the ring $\IntQ(\{\alpha\},\ohZ)$ is a PID.

\begin{Cor}\label{InQEZ PID}
Let $\underline{E}=\prod_p E_p\subset\overline{\hZ}$ be a subset  such that, for each $p\in\PP$, the elements of $E_p$ are pairwise non-conjugate over $\Q_p$. Then $\IntQ(\uE,\ohZ)$ is a PID  with finite residue fields of prime characteristic  if and only if, for each prime $p$, $E_p$ contains at most one element of $\overline{\Z_p}$, unramified over $\Q_p$ and transcendental over $\Q$ and $\uE$ is polynomially factorizable.
\end{Cor}
Note that if the conditions of the corollary occur, namely, $E_p=\{\alpha_p\}$ for each $p\in\PP$, then $\uE$ is the singleton $\{\alpha\}$, where $\alpha=(\alpha_p)\in\ohZ$.  The condition that $\uE$ is polynomially factorizable appears in another equivalent forms in \cite[Lemma 3.3]{GliSa} and \cite[Proposition 1.1]{GliSgaSa}, in the case $\alpha\in\hZ$.

\begin{proof} The proof follows from Theorem \ref{DedekindInt}, Lemma \ref{Picdirectsum}  and Proposition \ref{finiteZp}. 
\end{proof}

An argument similar to the one in the proof of \cite[Theorem]{EakHei} shows that a PID $\IntQ(\{\alpha\},\ohZ)$ as in the statement of Corollary \ref{InQEZ PID} is never an Euclidean domain.

We now show that each Dedekind domain  with finite residue fields of prime characteristic  between $\Z[X]$ and $\Q[X]$ is indeed a generalized ring of integer-valued polynomials.

\begin{Thm}\label{RDedekind}
Let $R$ be a Dedekind domain  with finite residue fields of prime characteristic  such that $\Z[X]\subset R\subseteq\Q[X]$. Then $R$ is equal to $\Int_{\Q}(\uE,\ohZ)$, for some subset $\uE=\prod_p E_p\subset\ohZ$ such that $E_p$ is a finite set of transcendental elements over $\Q$ for each prime $p$ and $\uE$ is polynomially factorizable.

In particular, the class group of $R$ is isomorphic to a direct sum of a countable family of finitely generated abelian groups.
\end{Thm}
\begin{proof}
Let $\PP_R=\{p\in \PP\mid \exists P\in\Spec(R) \text{ such that } P\cap\Z=p\Z\}$. Clearly, $\PP_R$ is empty if and only if $R=\Q[X]$; in this case for $\uE$ equal to the empty set we have the claim. Suppose $\PP_R$ is not  empty. For each $p\in\PP_R$, we denote by $\PP_{R,p}$ the set of unitary prime ideals of $R$ lying above $p$. By assumption,  for each $P\in\PP_{R,p}$, $p\in\PP$,  $R_P$ is a DVR of $\Q(X)$ with finite residue field extending $\Z_{(p)}$. In particular, by Theorem \ref{extension Vpa}, there exists $\alpha_p\in\oZp$, transcendental over $\Q$, such that $R_P=V_{p,\alpha_p}$. Let $E_p$ be the subset of $\overline{\Z_p}$ formed by such $\alpha_p$'s, for each $P\in\PP_{R,p}$. Since $R$ is Dedekind and by \eqref{IntQEZrepr} and \eqref{IntQEpZprepr}, we have the following equalities: 
$$R=\bigcap_{p\in\PP_R}\bigcap_{P\in\PP_{R,p}}R_P\cap\Q[X]=\bigcap_{p\in\PP_R}\bigcap_{\alpha_p\in E_p}V_{p,\alpha_p}\cap\Q[X]=\bigcap_{p\in\PP_R}\Int_{\Q}(E_p,\overline{\Z_p})=\Int_{\Q}(\uE,\ohZ)$$ 
where $\uE=\prod_{p\in\PP_R}E_p\subset\ohZ$. By Theorem  \ref{DedekindInt},  for each $p\in\PP$, $E_p$ is a finite subset of $\overline{\Z_p}$ of transcendental elements over $\Q$, $\uE$ is polynomially factorizable and the class group of $R$ is isomorphic to a direct sum of a countable family of finitely generated abelian groups.
\end{proof}

It has been showed in \cite[Proposition 3.4]{GliSa} that  the cardinality of the set of $\alpha\in\hZ$ such that $\IntQ(\{\alpha\},\hZ)$ is a PID is $2^{\aleph_0}$.  The next corollary describes all the PIDs  with finite residue fields of prime characteristic  between $\Z[X]$ and $\Q[X]$.
\begin{Cor}\label{RPID}
Let $R$ be a PID  with finite residue fields of prime characteristic  such that $\Z[X]\subset R\subset\Q[X]$. Then $R$ is equal to $\IntQ(\{\alpha\},\ohZ)$, for some $\alpha=(\alpha_p)\in\ohZ$ such that, for each $p\in \PP$, $\alpha_p$ is transcendental over $\Q$, $\alpha_p$ is unramified over $\Q_p$ and $\{\alpha\}$ is polynomially factorizable.
\end{Cor}
\begin{proof}
The proof follows from  Theorem \ref{RDedekind} and Corollary \ref{InQEZ PID}. 
\end{proof}

\subsection{Equality of generalized rings of integer-valued polynomials}

Given two locally bounded closed subsets $\uE,\underline{F}$ of $\ohZ$, we characterize when the associated generalized ring of integer-valued polynomials $\IntQ(\uE,\ohZ)$, $\IntQ(\underline{F},\ohZ)$ are the same.

The following is a general result about integral extensions of ring of integer-valued polynomials. For an integral domain $D$ with quotient field $K$, let $\olK$ and $\olD$ be the algebraic closure of $K$ and the absolute integral closure of $D$, respectively. We let $G_K=\Gal(\olK/K)$ be the absolute Galois group of $K$.  For a subset $\Omega$ of $\olK$ we set $G_K(\Omega)=\{\sigma(a)\mid \sigma\in G_K,a\in \Omega\}=\bigcup_{\sigma\in G_K}\sigma(\Omega)$. We say that $\Omega$ is \emph{$G_K$-invariant} if $G_K(\Omega)=\Omega$. Note that in general we have 
\begin{equation}\label{Galois polynomial closure}
\Int_K(\Omega,\olD)=\Int_K(G_K(\Omega),\olD)
\end{equation}
because if $f(\alpha)\in \olD$ for some $f\in K[X]$ and $\alpha\in\Omega$ , then, for every $\sigma\in G_K$, we have  $f(\sigma(\alpha))=\sigma(f(\alpha))\in\olD$ because $\sigma(\overline D)\subseteq\overline D$ . 
\begin{Lem}\label{integral extension}
Let $D$ be an integrally closed domain with quotient field $K$. Let $\Omega\subset\olD$ be $G_K$-invariant. Let $F$ be an algebraic extension of $K$ containing $\Omega$. Then $\Int_F(\Omega,\olD)$ 
is the integral closure of $\Int_K(\Omega,\olD)$ in $F(X)$.
\end{Lem}
\begin{proof}
By \cite[Proposition IV.4.1]{CaCh}, $\Int_{\overline{K}}(\Omega,\olD)$ is integrally closed. In particular, $\Int_F(\Omega,\olD)=\Int_{\overline{K}}(\Omega,\olD)\cap F(X)$ is integrally closed, too. Hence, we just need to show that $\Int_K(\Omega,\olD)\subseteq\Int_F(\Omega,\olD)$ is an integral ring extension.

Without loss of generality, we may enlarge $F$ and suppose that $F$ is normal over $K$ (e.g., we may take $F=\olK$). Let $f\in \Int_F(\Omega,\olD)\subset F[X]$. In particular, $f$ is integral over $K[X]$, i.e., it satisfies a monic equation of the form:
$$f^n+g_{n-1}f^{n-1}+\ldots+g_1 f+g_0=0$$
where $g_i\in K[X]$, for $i=0,\ldots,n-1$. We claim that $g_i\in \Int_K(\Omega,\olD)$, for $i=0,\ldots,n-1$, which will prove the claim. In fact, let $\Phi(T)=T^n+g_{n-1}T^{n-1}+\ldots+g_0\in K[X][T]$  and suppose that $\Phi(T)$ is irreducible over $K(X)$. The roots of $\Phi(T)$ are  the conjugates of $f$ under the action of the Galois group $\Gal(F(X)/K(X))\cong{\rm Gal}(F/K)$, which acts on the coefficients of the polynomial $f$. If $\sigma\in\Gal(F/K)$, then $\sigma(f)\in \Int_F(\Omega,\olD)$. In fact, for each $\alpha\in \Omega$, since $\Omega$ is ${\rm Gal}(F/K)$-invariant, we have $\alpha=\sigma(\alpha')$ for some $\alpha'\in\Omega$, therefore $\sigma(f)(\alpha)=\sigma(f(\alpha'))$ which still is an element of $\olD$ (which likewise is left invariant under the action of ${\rm Gal}(F/K)$). Now, since each coefficient $g_i$ of $\Phi(T)$ is an elementary symmetric function of  the elements $\sigma(f), \sigma\in\Gal(F/K)$, we have $g_i(\alpha)\in \olD$, for each $\alpha\in \Omega$, thus $g_i\in\Int_K(\Omega,\olD)$, as claimed.
\end{proof}
In order to ease the notation, for each prime $p$, we set $G_p=G_{\Q_p}$, the absolute Galois group of $\Q_p$.

\begin{Thm}\label{equality}
Let $\uE=\prod_p E_p,\underline{F}=\prod_p F_p$ be locally bounded closed subsets of $\ohZ$. Then $\IntQ(\uE,\ohZ)=\IntQ(\underline{F},\ohZ)$ if and only if $G_p(E_p)=G_p(F_p)$,  for each $p\in\PP$.
\end{Thm}
\begin{proof}
Clearly, $\IntQ(\uE,\ohZ)=\IntQ(\underline{F},\ohZ)$ if and only if the two rings have the same localization at each $p\in \PP$, that is, by Lemma \ref{localization}, $\IntQ(E_p,\overline{\Z_p})=\IntQ(F_p,\overline{\Z_p})$. Such a  condition is equivalent to $\Int_{\Q_p}(E_p,\overline{\Z_p})=\Int_{\Q_p}(F_p,\overline{\Z_p})$. In fact, one implication is obvious because $\IntQ(E_p,\overline{\Z_p})$ is the contraction to $\Q[X]$ of $\Int_{\Q_p}(E_p,\overline{\Z_p})$. Conversely, suppose that $\IntQ(E_p,\overline{\Z_p})=\IntQ(F_p,\overline{\Z_p})$ and let $f\in \Int_{\Q_p}(E_p,\overline{\Z_p})$, say $f(X)=\sum_i \alpha_i X^i$. We can choose $g\in\Q[X]$ sufficiently $v_p$-adically close to $f(X)$, that is, $g(X)=\sum_i a_i X^i$, where $v_p(\alpha_i-a_i)\geq n$ for each $i\geq0$, where $n\in\N$ is arbitrary large. In particular, $h=f-g\in p^n\Z_p[X]$, so, if $\alpha_p\in E_p$, it follows that $g(\alpha_p)=f(\alpha_p)+h(\alpha_p)\in\overline{\Z_p}$. Hence, $g\in\IntQ(E_p,\overline{\Z_p})=\IntQ(F_p,\overline{\Z_p})$. If now $\beta_p\in F_p$, we have $f(\beta_p)=g(\beta_p)+h(\beta_p)\in\overline{\Z_p}$, which proves that $f\in\Int_{\Q_p}(F_p,\overline{\Z_p})$. The other containment $\Int_{\Q_p}(F_p,\overline{\Z_p})\subseteq\Int_{\Q_p}(E_p,\overline{\Z_p})$ follows in the same way.

Let $p\in\PP$ be a fixed prime and set $\widehat{R}_{p,E_p}=\Int_{\Q_p}(E_p,\overline{\Z_p})$ and $\widehat{R}_{p,F_p}=\Int_{\Q_p}(F_p,\overline{\Z_p})$. Since $E_p, F_p$ are subsets of $\overline{\Z_p}$ of bounded degree, there exists a finite Galois extension $K$ of $\Q_p$ containing both of them. By \eqref{Galois polynomial closure}, $\widehat{R}_{p,E_p}=\Int_{\Q_p}(G_p(E_p),\overline{\Z_p})$ and $\widehat{R}_{p,F_p}=\Int_{\Q_p}(G_p(F_p),\overline{\Z_p})$. Clearly, $\widehat{R}_{p,E_p}$ and $\widehat{R}_{p,F_p}$ are equal if and only if they have the same integral closure in $K(X)$. By Lemma \ref{integral extension}, this amounts to say that 
\begin{equation}\label{equal int rings}
\Int_K(G_p(E_p),\overline{\Z_p})=\Int_K(G_p(F_p),\overline{\Z_p}).
\end{equation}
Note that the rings of the equality \eqref{equal int rings} are equal to $\Int_K(G_p(E_p),O_K)$, $\Int_K(G_p(F_p),O_K)$, respectively, where $O_K$ is the ring of integers of $K$. Moreover, $G_p(E_p)$ is a closed subset of $O_K$, being a finite union of closed sets $\sigma(E_p)$, $\sigma\in\Gal(K/\Q_p)$. Similarly, $G_p(F_p)$ is closed.

Finally, by \cite[Lemma 2]{McQ Gilmer}, equation \eqref{equal int rings} holds if and only if  $G_p(E_p)=G_p(F_p)$.\end{proof}  

In particular, Theorem \ref{equality} implies that  the rings $\IntQ(\{\alpha\},\hZ)$, $\alpha\in\hZ$, are in one-to-one correspondence with the elements of $\hZ$.

\section{Construction of a Dedekind domain with prescribed class group}\label{construction}

We review  Chang's construction in \cite{Chang} that we mentioned in the Introduction and modify it in order to show that, given a group $G$ which is the direct sum of a countable family of finitely generated abelian groups, there exists a Dedekind domain $R$  with finite residue fields of prime characteristic, $\Z[X]\subset R\subseteq \Q[X]$  such that the class group of $R$ is $G$. As in Eakin and Heinzer's result in \cite{EakHei}, we show first that the ring constructed by Chang  can also be represented as a generalized ring of integer-valued polynomials. In \cite[Lemma 3.4]{Chang} it is proved that for each $n\in\N$ and $p\in\PP$, there exists a DVR of $\Q(X)$ which is a residually algebraic extension of $\Z_{(p)}$ with ramification index equal to $n$; by means of Theorem \ref{extension Vpa}, we can give an explicit representation of such an extension in terms of a  valuation domain $V_{p,\alpha}$ associated to some $\alpha\in\overline{\Z_p}$ which generates a  totally ramified extension of $\Q_p$ of degree $n$.

Let $I$ be a countable set and  $G=\bigoplus_{i\in I}G_i$ be a direct sum of finitely generated abelian groups $G_i$. Suppose that for each $i\in I$ we have
$$G_i\cong\Z^{m_i}\oplus\Z/n_{i,1}\Z\oplus\ldots\oplus\Z/n_{i,k_i}\Z$$
for some uniquely determined non-negative integers $m_i,n_{i,1},\ldots,n_{i,k_i}$ with $n_{i,j}\mid n_{i,j+1}$. We partition $\PP$ into a family of finite subsets $\{\PP_i\}_{i\in I}$ each of which contains arbitrary chosen $1+k_i$ primes, namely $\PP_i=\{p_i,q_{i,1},\ldots,q_{i,k_i}\}$ and correspondingly for each $i\in I$ we fix the following $1+k_i$ sets:
\begin{itemize}
\item[i)] $E_{p_i}$ is a subset of $\Z_{p_i}$  of $m_i+1$ elements $\{\alpha_{p_i,1},\ldots,\alpha_{p_i,m_i+1}\}$ which are transcendental over $\Q$.
\item[ii)] for $j=1,\ldots,k_i$, $E_{q_{i,j}}=\{\alpha_{q_{i,j}}\}$ a singleton of $\overline{\Z_{q_{i,j}}}$ such that $\alpha_{q_{i,j}}$ is transcendental over $\Q$ and $n_{i,j}=e_{\alpha_{q_{i,j}}}$, the ramification index of $\Q_p(\alpha_{q_{i,j}})$ over $\Q_p$.
\end{itemize}
We set $\underline E_i=E_{p_i}\times\prod_{j=1}^{k_i}E_{q_{i,j}}$ and also 
$$R_i=\IntQ(E_{p_i},\Z_{p_i})\cap\bigcap_{j=1}^{k_i}\IntQ(E_{q_{i,j}},\overline{\Z}_{q_{i,j}})=\IntQ(\underline E_i,\overline{\hZ}).$$
Since each of the unitary valuation overrings of $R_i$, namely $V_{p,\alpha_p}$, $p\in\PP_i$ and $\alpha_p\in E_p$, is a DVR which is residually algebraic over $\mathbb{F}_{p}$ (\cite[Proposition 2.2]{PerTransc}),  by \cite[Theorem \& Corollary]{EakHei} $R_i$ is a Dedekind domain  with class group isomorphic to $G_i$.

We also set
$$R=\bigcap_{i\in I}R_i=\Int_{\Q}(\underline{E},\overline{\hZ})$$
where $\underline{E}=\prod_i \underline{E}_i$. By \cite[Corollary 2.6]{Chang}, $R$ is an almost Dedekind domain with class group isomorphic to $G$. 

As we already mentioned at the beginning of \S \ref{generalized not Noetherian}, the ring $R=\Int_{\Q}(\uE,\ohZ)$ is not Dedekind in general. By Theorem \ref{DedekindInt}, this happens precisely when  $\uE$ is polynomially factorizable. By a suitable modification of the above construction, we are going to show that there exists a polynomially factorizable subset $\uE$ of $\ohZ$  such that $R$ is Dedekind   with class group isomorphic to $G$, thus giving a positive answer to  \cite[Question 3.7]{Chang}.

\begin{Thm}
Let $G$ be a direct sum of a countable family $\{G_i\}_{i\in I}$ of  finitely generated abelian groups (which are not necessarily distinct). Then there exists a Dedekind domain $R$ between $\Z[X]$ and $\Q[X]$ with class group isomorphic to $G$. Moreover, for each $i\in I$, there exists a multiplicative subset $S_i$ of $\Z$ such that $S_i^{-1}R$ is a Dedekind domain with class group $G_i$.
\end{Thm}
\begin{proof} We keep the notation used in the above construction. Let $\PP_r=\bigcup_{i\in I}(\PP_i\setminus \{p_i\})$. For each $q=q_{i,j}\in\PP_r$, for some $i\in I$ and $j\in\{1,\ldots,k_i\}$, we set  $n_q=n_{i,j}$. We choose  a uniformizer $\tilde{q}$ of $\Z_q$ which is transcendental over $\Q$. Let $\tilde\alpha_q\in\overline{\Z_q}$ be a root of the Eisenstein polynomial $X^{n_q}-\tilde{q}$. Clearly, $\tilde\alpha_q$ is still transcendental over $\Q$ and it is well-known that $\Q_q(\tilde\alpha_q)$ is a totally ramified extension of $\Q_q$ of degree $n_q$. We let now 
$\alpha_q=\tilde\alpha_q+\lfloor\log q\rfloor$: this is another generator of $\Q_q(\tilde\alpha_q)$ over $\Q_q$ which still is transcendental over $\Q$ and has $v_q$-adic valuation zero.  We then set $E_q=\{\alpha_q\}$ in the above construction.

Similarly, for each $p=p_i\in\PP\setminus\PP_r$, for some $i\in I$, let $m_p=m_{p_i}$. We choose distinct elements $\alpha_{p,i}\in \lfloor\log p\rfloor+p\Z_p$, for $i=1,\ldots,m_p+1$, which are transcendental over $\Q$ and set $E_p=\{\alpha_{p,1},\ldots,\alpha_{p,m_p+1}\}$.

We show now that with these choices    the subset $\uE=\prod_p E_p\subset\ohZ$ is polynomially factorizable, and therefore the corresponding domain $R=\IntQ(\uE,\ohZ)$ is a Dedekind domain by Theorem \ref{DedekindInt}. By Lemma \ref{pol factorizable}, we need to show that for each $g\in\Z[X]$, $\PP_{g,\uE}$ is finite. Let $g\in\Z[X]$ be a fixed polynomial. For  $\alpha=(\alpha_p)\in \uE$, we have: 
\begin{itemize}
\item[-] $\alpha_p=pa+\lfloor\log p\rfloor$, for some $a\in\Z_p$, if $p\in\PP\setminus\PP_r$.
\item[-] $\alpha_p=\tilde\alpha_p+\lfloor\log p\rfloor$, if $p\in\PP_r$,  where $\tilde\alpha_p$ is a root of an Eisenstein polynomial, so, in particular,  $v_p(\tilde\alpha_p)>0$.
\end{itemize}
For each $p\in\PP$, let $\pi_p$ be a uniformizer of $\Q_p(\alpha_p)$ (which is just $p$ if $p\notin\PP_r$). We then have
$$g(\alpha_p)\equiv g(\lfloor\log p\rfloor)\pmod {\pi_p}.$$
Now, for all $p$ sufficiently large, $g(\lfloor\log p\rfloor)$ is not divisible by $p$, since $\lim_{x\to\infty}\frac{g(\log x)}{x}=0$. Hence, $\PP_{g,\uE}$ is finite.

The fact that $\IntQ(\uE,\ohZ)$ has class group equal to $G$ follows either by \cite[Corollary 2.6]{Chang} or by applying Lemma \ref{Picdirectsum} and Proposition \ref{finiteZp}, by noting that $\Pic(\IntQ(E_p,\overline{\Z_p}))=\Z^{m_p}$ for each $p\in\PP\setminus\PP_r$ and $\Pic(\IntQ(E_q,\overline{\Z_q}))=\Z/n_q\Z$ for each $q\in\PP_r$.

For the last claim, if $i\in I$, we let $S_i$ be the multiplicative subset of $\Z$ generated by $\PP\setminus\PP_i$. Then, by Lemma \ref{localization}, $S_i^{-1}\Int_{\Q}(\uE,\ohZ)=\Int_{\Q}(\underline{E}_i,\ohZ)$  which has class group isomorphic to $G_i$ by Lemma \ref{Picdirectsum} and Proposition \ref{finiteZp}. 
\end{proof}

\subsection*{Acknowledgments}

The author wishes to thank Jean-Luc Chabert for reading a preliminary version of this paper.  The author also wishes also to thank the anonymous referees for their remarks.

\end{document}